\documentclass[leqno,12pt]{amsart}
\usepackage{stmaryrd} 
\usepackage{amssymb}
\theoremstyle{plain} 
\newtheorem{theorem}{\indent\sc Theorem}[section]
\newtheorem{lemma}[theorem]{\indent\sc Lemma}

\theoremstyle{definition} 

\newtheorem{remark}[theorem]{\indent\sc Remark}

\begin{document}

\title[ 3D incompressible Navier-Stokes equations with damping ]{Asymptotic behavior of solutions to 3D incompressible Navier-Stokes equations with damping} 

\author[X. Zhao]{Xiaopeng Zhao$^*$} 
\author[H. Meng]{Haichao Meng}


\subjclass[2010]{ 
35B40, 35L70, 35Q72,
}
%
\keywords{ 
Navier-Stokes equations, generalized Navier-Stokes equations, damping, decay rate.
}
\thanks{ 
$^{*}$This paper is supported by the NSFC (grant
No. 11401258), NSF of Jiangsu Province (grant
No. BK20140130) and China Postdoctoral Science Foundation (grant No. 2015M581689
).}
\address{ Xiaopeng Zhao, Haichao Meng \endgraf
School of Science\endgraf
Jiangnan University\endgraf
Wuxi 214122,~~~P. R. China\endgraf
}
\email{X. Zhao: zhaoxiaopeng@jiangnan.edu.cn}


\begin{abstract}
In this paper,  we study the upper   bound of the time decay rate of solutions to the   Navier-Stokes equations and generalized Navier-Stokes equations with damping term $|u|^{\beta-1}u$ ($\beta>1$) in $\mathbb{R}^3$.
\end{abstract}\maketitle

\section{Introduction} \label{wsect1}

We study  the  Cauchy problem for the 3D
  Navier-Stokes equations with damping
  \begin{equation} \label{w1-0}
\left\{ \begin{aligned}
       & \frac{\partial u}{\partial t}+u\cdot\nabla u+\nabla p-\Delta u+\nu|u|^{\beta-1}u=0, \\
               &   \nabla\cdot u=0,\\
                &  u(x,0)=u_0(x),
                          \end{aligned} \right.
                          \end{equation}
                and the 3D generalized Navier-Stokes equations with damping
\begin{equation} \label{w1-1}
\left\{ \begin{aligned}
       & \frac{\partial u}{\partial t}+u\cdot\nabla u+\nabla p+(-\Delta)^{\alpha}u+\nu|u|^{\beta-1}u=0, \\
               &   \nabla\cdot u=0,\\
                &  u(x,0)=u_0(x),
                          \end{aligned} \right.
                          \end{equation}
                          where $u=u(x,t)=(u_1(x,t),u_2(x,t),u_3(x,t))\in\mathbb{R}^3$ and $p=p(x,t)\in\mathbb{R}$ are unknown velocity field and pressure respectively.
The above two systems, which describe  porous media flow, friction effects or some dissipative mechanisms and so on, are of interest for various reasons.   The fractional power of the Laplace transform $(-\Delta)^{\alpha}\equiv\Lambda^{2\alpha}$ is defined through Fourier transform (see\cite{ 21})
$$
\widehat{(-\Delta)^{\alpha}f}(\xi)=\widehat{\Lambda^{2\alpha}f}(\xi)=|\xi|^{2\alpha}\widehat{f}(\xi),\quad\widehat{f}(\xi)=\int_{R^3}f(x)e^{-2\pi ix\cdot\xi}dx.
$$


 It was Cai and Jiu\cite{Jiu}
who first gave the physical background and studied the well-posedness of system (\ref{w1-0}) in 3D case. Later, 
Zhou\cite{ZY} improved the results in \cite{Jiu}. There are also  some papers concerned with the decay of solutions for system (\ref{w1-0}). By using Fourier splitting method(see~\cite{S1,S2}),  Cai and Lei\cite{CL} proved that if $u_0\in L^1(\mathbb{R}^3)\bigcap L^2(\mathbb{R}^3)$ and $\beta>\frac73$, the decay of weak solutions has a uniform rate
\begin{equation}\label{wadd1}
\|u\|_{L^2}^2\leq C(1+t)^{-\min\{\frac12,\frac{3\beta-7}{2(\beta+1)}\}}.
\end{equation}Jia, Zhang and Dong\cite{Jia} supposed that $\beta\geq\frac{10}3$  and $\|e^{\Delta t}u_0\|^2\leq C(1+t)^{-\mu}$~($\mu>0$), obtained the $L^2$ decay rate for system (\ref{w1-0}):
\begin{equation}\label{wadd2}
\|u\|_{L^2}^2\leq C(1+t)^{-\min\{\mu,\frac32\}}.
\end{equation}
In addition, Jiang\cite{JZH} used Fourier splitting method to extend the above two results, obtained the decay estimate
\begin{equation}\label{wadd3}
\|u\|_{L^2}^2\leq C(1+t)^{-\frac32},
\end{equation}for $u_0\in H^1(\mathbb{R}^3)$ and $\beta\geq3$. Latterly,      by using Zhou's method(see~\cite{Zhou1}), Jiang and Zhu\cite{JZ} established the  decay rate (\ref{wadd2}) for $u_0\in H^1(\mathbb{R}^3)$ and $\beta\geq3$.

It is worth pointing out that if $\nu=0$ in system (\ref{w1-1}), we obtain the generalized Navier-Stokes equations\cite{Wu,Wu2}. In \cite{Jiu}, Jiu and Yu studied the decay of solutions to the 3D generalized Navier-Stokes equations. Supposed that $0<\alpha<\frac54$ and $u_0\in L^2(\mathbb{R}^3)\bigcap L^p(\mathbb{R}^3)$ with $\max\{1,\frac1{3-2\alpha}\}\leq p<2$, the authors showed that the decay of the solution is
\begin{equation}
\label{wadd4}
\|u\|_{L^2}^2\leq C(1+t)^{-\frac3{2\alpha}\left(\frac2p-1\right)}.
\end{equation}
 Recently, Duan \cite{Duan} improved the above result to the case $0<\alpha<2$ and $u_0\in L^2(\mathbb{R}^3)\bigcap L^1(\mathbb{R}^3)$, proved that
\begin{equation}
\label{wadd5}
\|u\|_{L^2}^2\leq C(1+t)^{-\frac3{2\alpha} }.
\end{equation}

Analyze the above results, we find that  there's no result on the systems (\ref{w1-0}) and  (\ref{w1-1}) with $\beta<3$.  Can we establish the decay rate of solutions for systems (\ref{w1-0}) and  (\ref{w1-1}) when $\beta\in[1,3)$?  It seems as an interesting question.

In this paper, we consider the  time decay rate of solutions to systems (\ref{w1-0}) and (\ref{w1-1}). The motivation is to understand how the parameter $\beta$ affect the time decay rate of its solutions.
 Here, we contrast with the generalized heat equations,  study the decay rate of solutions for systems (\ref{w1-0}) and (\ref{w1-1}), establish the $L^2$ decay of solutions for  $u_0\in L^2(\mathbb{R}^3)\bigcap L^1(\mathbb{R}^3)$ and $\beta\geq1$. More precisely, we have the following two results

 \begin{theorem}
\label{wthm1.1}
 Suppose that $\beta\geq1$, $u_0\in L^2(\mathbb{R}^3)\bigcap L^1(\mathbb{R}^3)$ and $\nabla\cdot u_0=0$.  Then, for the solution $u(x,t)$ of system (\ref{w1-0}), there is a positive constant $C=C(\beta,\|u_0\|_{L^1},\|u_0\|_{L^2})$, such that
 $$
 \|u(x,t)\|_{L^2}^2\leq C(1+t)^{-\min\{\frac32,\frac{3\beta-2 }{2 }\}}~~\hbox{for~large}~t.
 $$
\end{theorem}

\begin{theorem}
\label{wlem2.2}
 Suppose that $0<\alpha<\frac54$, $\beta\geq1$,  $u_0\in L^2(\mathbb{R}^3)\bigcap L^1(\mathbb{R}^3)$ and $\nabla\cdot u_0=0$.  Then, for the solution $u(x,t)$ of system (\ref{w1-1}), there exists a positive constant $C=C(\alpha,\beta,\|u_0\|_{L^1},\|u_0\|_{L^2})$, such that
 $$
 \|u(x,t)\|_{L^2}^2\leq C(1+t)^{-\min\{\frac3{2\alpha},\frac{3\beta-2\alpha }{2 \alpha}\}}~~\hbox{for~large}~t.
 $$
\end{theorem}

\begin{remark}
Fourier splitting method is introduced by Schonbek in 1980s (see~\cite{S1,S2}), then it becomes a standard way (also a powerful tool) to establish decay rate of solutions. In 2007, Zhou introduced a new method (see~\cite{Zhou1}, some people called Zhou's method) to handle decay rate problems. One can refer to~\cite{Zhou1,Zhou2,JZH,JZ} for details and developments.
\end{remark}
\begin{remark}
It is important to note that
 because of the existence of the   damping term $|u|^{\beta-1}u$, if we use Fourier splitting method or Zhou's method to study the decay estimate, we have a problem: How to control $\|u\|_{\beta}^{\beta}$? Of course, the interpolation inequality is one of the effective methods, we can use $\|u\|^2_{L^2}$ and $\|u\|_{L^{\beta+1}}^{\beta+1}$ to estimate this term. But, in order to let the estimate valid, we have to suppose that $\beta\geq 2$. 
 \end{remark}

 \begin{remark}
 As we know, Theorems \ref{wthm1.1} and \ref{wlem2.2} are the first decay estimates for  systems (\ref{w1-0}) and (\ref{w1-1}) with $\beta\in[1,\infty)$.
 \end{remark}

 \section{Proof of Theorem \ref{wlem2.2} }

Since Theorem \ref{wthm1.1} is one special case of Theorem \ref{wlem2.2}. 
We only prove Theorem \ref{wlem2.2} in this paper.

For the Cauchy problem of generalized heat equation
\begin{equation}
\begin{aligned}\label{w2-1}
v_t+(-\Delta)^{\alpha}v=0,~~~
v(x,0)=u_0(x),\end{aligned}
\end{equation}
we have the following space-time estimates:
\begin{lemma}[see \cite{Miao}]\label{wlem2.1}
Let $1\leq r\leq q\leq\infty$ and $u_0(x)\in L^r(\mathbb{R}^3)$. Then, for $\alpha>0$ and $\mu>0$, problem (\ref{w2-1}) satisfies the following estimates:
\begin{equation}
\label{w2-2}
\|v(x,t)\|_{L^q}\leq Ct^{-\frac 3{2\alpha}(\frac1r-\frac1q)}\|u_0\|_{L^r},
\end{equation}
and
\begin{equation}
\label{w2-3}
\|(-\Delta)^{-\frac{\mu}2}v(x,t)\|_{L^q}\leq Ct^{-\frac{\mu}{2\alpha}-\frac3{2\alpha}(\frac1r-\frac1q)}\|u_0\|_{L^r}.
\end{equation}

\end{lemma}
\begin{remark}
By (\ref{w2-2}),  if  $u_0(x)\in L^1(\mathbb{R}^3)$,
\begin{equation}\label{w2-2a}\|v\|_{L^2}\leq Ct^{-\frac34}\|u_0\|_{L^1}.
\end{equation}
\end{remark}


Take $w(x,t)=u(x,t)-v(x,t)$. Then $w$ solves the equations 
\begin{equation} \label{w2-4}
\left\{ \begin{aligned}
         &\partial_tw+\Lambda^{2\alpha}w+u\cdot\nabla u+\nabla p+|u|^{\beta-1}u=0,\\
                  &\nabla\cdot w=0,\\
                  &w(x,0)=0.
                          \end{aligned} \right.
                          \end{equation}

Now, we give the proof of Theorem \ref{wlem2.2}.
\begin{proof}[Proof of Theorem \ref{wlem2.2}]
Multiplying both sides of Equation (\ref{w2-4}) by $w$, integrating over $\mathbb{R}^3$, we deduce that
\begin{equation}
\label{w2-8}\begin{aligned}&
\frac12\frac d{dt}\|w\|^2_{L^2}+\|\Lambda^{\alpha}w\|_{L^2}^2\\=&-\sum_{i,j=1}^3\int_{\mathbb{R}^3}u_i\frac{\partial u_j}{\partial x_i}w_jdx-\int_{\mathbb{R}^3}w\cdot\nabla pdx-\int_{\mathbb{R}^3}|u|^{\beta-1}u\cdot wdx.
\end{aligned}\end{equation}
We have
\begin{equation}
\begin{aligned}\label{w2-9}
\left|-\sum_{i,j=1}^3\int_{\mathbb{R}^3}u_i\frac{\partial u_j}{\partial x_i}w_jdx\right|
=&\left|\sum_{i,j=1}^3\int_{\mathbb{R}^3}u_i\frac{\partial u_j}{\partial x_i}(u_j-v_j)dx\right|=\left|\sum_{i,j=1}^3\int_{\mathbb{R}^3}u_i\frac{\partial u_j}{\partial x_i}v_jdx\right|
\\
=&\left|\sum_{i,j=1}^3\int_{\mathbb{R}^3}u_i\frac{\partial v_j}{\partial x_i}u_jdx\right|\leq C\|\nabla v\|_{L^{\infty}}\|u\|_{L^2}^2.
\end{aligned}\end{equation}
and
\begin{equation}
\label{w2-9-1}\begin{aligned}
-\int_{\mathbb{R}^3}|u|^{\beta-1}u\cdot wdx=&-\int_{\mathbb{R}^3}|u|^{\beta-1}u\cdot (u-v)dx\\=&-\|u\|_{L^{\beta+1}}^{\beta+1}+\int_{\mathbb{R}^3}|u|^{\beta-1}u\cdot vdx
\\\leq& -\frac12\|u\|_{L^{\beta+1}}^{\beta+1}+C\|v\|_{L^{\beta+1}}^{\beta+1}.\end{aligned}
\end{equation}
Note that $\nabla\cdot u=\nabla
\cdot w=0$. We have
\begin{equation}
\label{w2-10}
\int_{\mathbb{R}^3}w\cdot\nabla pdx=0.\end{equation}
Combining (\ref{w2-8})-(\ref{w2-10}) together gives
\begin{equation}
\label{w2-11}
\frac d{dt}\|w\|_{L^2}^2+2\|\Lambda w\|_{L^2}^2+\|u\|_{L^{\beta+1}}^{\beta+1}\leq \|\nabla v\|_{L^{\infty}}\|u\|_{L^2}^2+C\|v\|_{L^{\beta+1}}^{\beta+1}.
\end{equation}
By (\ref{w2-2}) and (\ref{w2-3}), we have the following estimates
\begin{equation}
\label{w2-13}
\|\nabla v\|_{L^{\infty}}\leq Ct^{-\frac2{\alpha}},~~~\|v\|_{L^{\beta+1}}\leq Ct^{-\frac3{2\alpha }\left(1-\frac1{\beta+1}\right)}.
\end{equation}
From (\ref{w2-11})-(\ref{w2-13}), we obtain
\begin{equation}\label{w2-14-1}\begin{aligned}
\frac d{dt}\|w\|_{L^2}^2+2\|\Lambda w\|_{L^2}^2+\|u\|_{L^{\beta+1}}^{\beta+1}
\leq&  C\left(\frac{t+t}2\right)^{-\frac2{\alpha}}\|u\|_{L^2}^2+C\left(\frac{t+t}2\right)^{-\frac{3\beta}{2\alpha }}
\\
\leq&C\left(\frac{1+t}2\right)^{-\frac2{\alpha}}\|u\|_{L^2}^2+C\left(\frac{1+t}2\right)^{-\frac{3\beta}{2\alpha }}
\\
\leq&
C(1+t)^{-\frac2{\alpha}}\|u\|_{L^2}^2+C(1+t)^{-\frac{3\beta}{2\alpha }},~~\forall t>1.\end{aligned}
\end{equation}
Note that $\alpha\in(0,\frac54)$  and $\|u\|_{L^2}\leq C$. Just as \cite{JZH}, using Fourier splitting method, we have
\begin{equation}
\label{w2-15}\begin{aligned}
\|w\|_{L^2}^2\leq&  C(1+t)^{-\min\{\frac{2-\alpha }{\alpha },\frac{3\beta-2\alpha }{2\alpha }\}},\quad\forall t>1.
\end{aligned}\end{equation}
Then, (\ref{w2-2a}) and (\ref{w2-15}) gives
\begin{equation}
\label{wx-1}\begin{aligned}
\|u\|_{L^2}^2=&\|v+w\|_{L^2}^2\leq C\|v\|_{L^2}^2+C\|w\|_{L^2}^2
\\\leq& C(1+t)^{-\frac3{2\alpha}}+C(1+t)^{-\min\{\frac{2-\alpha }{\alpha },\frac{3\beta-3-2\alpha }{2\alpha }\}}
\\
\leq&C(1+t)^{-\min\{\frac3{2\alpha},\frac{2-\alpha }{\alpha },\frac{3\beta-2\alpha }{2\alpha }\}},\quad\forall t>1.
\end{aligned}\end{equation}
We now use this first preliminary decay to bootstrap, trying to find sharper estimates.
Combining (\ref{w2-14-1}) and (\ref{wx-1}) together gives
\begin{equation}
\label{w912-1}
\begin{aligned}&
\frac d{dt}\|w\|_{L^2}^2+2\|\Lambda w\|_{L^2}^2+\|u\|_{L^{\beta+1}}^{\beta+1}\\
\leq& C(1+t)^{-\frac2{\alpha}}(1+t)^{-\min\{\frac3{2\alpha},\frac{2-\alpha }{\alpha },\frac{3\beta-2\alpha }{2\alpha }\}}+C(1+t)^{-\frac{3\beta}{2\alpha }},~~\forall t>1.\end{aligned}
\end{equation}
Hence
\begin{equation}\label{wx-2-1}\begin{aligned}
\|w\|_{L^2}^2\leq& C (1+t)^{-\frac{2-\alpha}{\alpha}}(1+t)^{-\min\{\frac3{2\alpha},\frac{2-\alpha }{\alpha },\frac{3\beta-2\alpha }{2\alpha }\}}+C(1+t)^{-\frac{3\beta-2\alpha}{2\alpha }}
\\
\leq&C(1+t)^{-\min\{\frac{3}{2\alpha},\frac{8-4\alpha}{\alpha},\frac{3\beta-2\alpha}{2\alpha}\}}
\leq C(1+t)^{-\min\{\frac{3}{2\alpha},\frac{3\beta-2\alpha}{2\alpha}\}},~~\forall t>1.\end{aligned}
\end{equation}
So(\ref{w2-2a}) 
and (\ref{wx-2-1}) gives
\begin{equation}
\label{wx-3}\begin{aligned}
\|u\|_{L^2}^2=&\|v+w\|_{L^2}^2\leq C\|v\|_{L^2}^2+C\|w\|_{L^2}^2
\\\leq& C(1+t)^{-\frac3{2\alpha}}+ C(1+t)^{-\min\{\frac3{2\alpha},\frac{3\beta-2\alpha}{2\alpha}\}}
\\
\leq&C(1+t)^{-\min\{\frac{3}{2\alpha},\frac{3\beta-2\alpha}{2\alpha}\}},\quad\forall t>1.
\end{aligned}\end{equation}
If we bootstrap  once again, the decay rate is also the same to (\ref{wx-3}), there is no improvement. Then, the estimate (\ref{wx-3}) is optimal, this concludes the proof of Theorem \ref{wlem2.2}.
\end{proof}

\section*{Acknowledgements}
This work is partially supported by  NSFC (Grant No. 11401258),  NSF of Jiangsu Province of China (Grant No. BK20140130)  and China Postdoctoral Science Foundation (Grant No. 2015M581689).

\end{document}